\documentclass[letterpaper, 10 pt, conference]{IEEEconf}  %

\IEEEoverridecommandlockouts                              %

\let\labelindent\relax 
\usepackage{amsmath}
\usepackage{amssymb}
\usepackage{amsthm}
\usepackage{xcolor}
\usepackage{enumitem}
\usepackage{hyperref}
\hypersetup{
    colorlinks,
    linkcolor={blue!50!black},
}

\newcommand{\R}{\mathbb{R}} %
\newcommand{\Rnn}{\mathbb{R}_{\geq 0}} %
\newcommand{\Rp}{\mathbb{R}_{>0}} %
\newcommand{\Np}{\mathbb{N}_+} %

\newcommand{\Indicator}[1]{\mathbb{I}\{{#1}\}}

\newcommand{\Borel}{\mathcal{B}} %
\newcommand{\Pbs}{\mathcal{P}} %
\newcommand{\E}{\mathbb{E}} %

\newcommand{\diff}{\, \mathrm{d}} %

\newcommand{\weak}{\rightharpoonup} %

\newcommand{\cKi}{\mathcal{K}_\infty} %

\newcommand{\s}{\phi} %
\newcommand{\es}{\bar{x}} %
\newcommand{\is}[1]{{\underline{#1}}} %

\theoremstyle{theorem}
\newtheorem{theorem}{Theorem}
\newtheorem{lemma}[theorem]{Lemma}

\newtheorem{corollary}[theorem]{Corollary}

\theoremstyle{definition}
\newtheorem{definition}{Definition}
\newtheorem{assumption}{Assumption}
\newtheorem{example}{Example}

\theoremstyle{remark}
\newtheorem{remark}{Remark}

\newif\ifInitialSubmission
\InitialSubmissiontrue %

\title{\LARGE \bf
Towards optimal control of ensembles of discrete-time systems
}

\author{Christian Fiedler$^{1}$ and Alessandro Scagliotti$^{1}$%
\thanks{$^{1}$Department of Mathematics, CIT School, Technical University of Munich, Germany, and Munich Center for Machine Learning
        {\tt\small christian.fiedler@tum.de, scag@ma.tum.de}%
}}

\begin{document}

\maketitle
\thispagestyle{empty}
\pagestyle{empty}

\begin{abstract}
The control of ensembles of dynamical systems is an intriguing and challenging problem, arising for example in quantum control.
We initiate the investigation of optimal control of ensembles of discrete-time systems, focusing on minimising the average finite horizon cost over the ensemble.
For very general nonlinear control systems and stage and terminal costs, we establish existence of minimisers under mild assumptions.
Furthermore, we provide a $\Gamma$-convergence result which enables consistent approximation of the challenging ensemble optimal control problem,
for example, by using empirical probability measures over the ensemble.
Our results form a solid foundation for discrete-time optimal control of ensembles, with many interesting avenues for future research.
\end{abstract}

\section{Introduction} \label{sec:intro}
In some control applications, the intriguing problem of controlling \emph{simultaneously} a whole ensemble of uncoupled dynamical systems arises.
This setup appears typically in two situations. 
First, one might have a collection of dynamical systems that can only be controlled collectively, i.e., the same control input is applied to all systems in the ensemble simultaneously. This happens for example in some quantum control problems \cite{li2006control,ruths2012optimal} or biomedical applications \cite{scagliotti2025ensemble}.
Second, one might have a single control system subject to some parametric uncertainty, modelled for example by a probability measure over the uncertain parameters.
Controlling this single, but uncertain control system can then be interpreted as controlling a whole ensemble of systems, one for each instance of the uncertainty \cite{Ross2015Riemann,Caillau2018Solving}.
Since the introduction of ensemble control setup in \cite{Li2009Ensemble}, this problem has received considerable attention from the control community.
A particular focus has been on controllability questions, which are challenging even in the linear case \cite{Loheac2016From,Baparou2024Conditions}, and which are of particular interest in quantum control \cite{Augier2018Adiabatic}.
Complementary to these issues, \cite{scagliotti2023optimal} has started to consider optimal control of ensembles of dynamical systems.
Conceptually, a cost functional is introduced for each system in the ensemble separately, and the costs are then aggregated over the whole ensemble.
In \cite{scagliotti2023optimal}, a distribution over the systems is assumed and the integral cost is optimised, corresponding to controlling the average cost of the ensemble, 
while in \cite{scagliotti2025minimax}, the supremum over the individual costs is used for aggregation, corresponding to minimax control of the ensemble.
In both cases, continuous-time control-affine systems over a finite horizon are considered.

In this work, we initiate the investigation of optimal control of ensembles of \emph{discrete-time} dynamical systems.
From a practical perspective, this is motivated by digital control systems and techniques like model predictive control,
and from a mathematical perspective, it is interesting since considerably more general systems and weaker regularity assumptions can be used.
Indeed, we consider general nonlinear control systems on general metric spaces, which is particularly interesting for distributed-parameter, hybrid, and digital control systems, and even the cost functions need only mild regularity assumptions, allowing a considerable degree of freedom in the choice thereof.
While some works have considered control of ensembles of discrete-time systems, e.g., \cite{tie2016controllability}, optimal control in a setting similar to \cite{scagliotti2023optimal} has not been investigated in a discrete-time setting before.
As a starting point, we focus on optimal control of the average finite horizon cost over the ensemble, as introduced in \cite{scagliotti2023optimal},
and we establish the existence of minimising control sequences using the direct method of the calculus of variations under rather mild assumptions, cf. Theorem \ref{thm:dteoc:existenceOfMinimisers}.
Since optimisation of an integral functional can be challenging, cf. Section \ref{sec:ecc:gammaConv} for some intuition on this, we consider approximations of the average cost optimal control problem, and we use $\Gamma$-convergence arguments to ensure the consistency of these approximations, cf. Theorem \ref{thm:dteoc:GammaConvergence} and Corollary \ref{cor:ecc:convOfMin}.
While this approach has been introduced in \cite{scagliotti2023optimal}, our general discrete-time setup differs significantly on a technical level.

\emph{Outline} In Section \ref{sec:ecc:prelims}, we introduce our notation, the problem setup, and provide some technical preliminaries.
We show existence of minimisers of the ensemble optimal control problem in Section \ref{sec:ecc:existenceMinimisers},
and in Section \ref{sec:ecc:gammaConv}, we establish a $\Gamma$-convergence result for the optimal control problems.
Section \ref{sec:ecc:conclusion} concludes the manuscript.
\section{Preliminaries} \label{sec:ecc:prelims}
\subsection{Notation}
For a metric space $(S,d_S)$, we denote by $\Borel(S)$ the Borel $\sigma$-algebra generated by the open sets of $S$,
and we denote by $\Pbs(S)$ the set of Borel probability measures on $S$.
If $\varphi \colon S \rightarrow \Rnn$ is $\Borel(S)$-$\Borel(\Rnn)$ measurable and $\mu \in \Pbs(S)$, we denote by $\int_S \varphi(s) \diff \mu(s)$ the corresponding Lebesgue integral.
We say that $(\mu_k)_k\subseteq\Pbs(S)$ converge weakly to $\mu\in\Pbs(S)$, written $\mu_k \rightharpoonup \mu$, if
\begin{equation*}
    \int_\Theta \varphi(s) \diff \mu_k(s) \rightarrow \int_\Theta \varphi(s) \diff \mu(s)
\end{equation*}
for all $\varphi \colon S\rightarrow\R$ bounded and continuous.
If $Y$ is a nonnegative random variable, $\E[Y]$ denotes its expectation.
The \emph{indicator function} for a predicate $P$ is defined by
$\Indicator{P}=1$ if $P$ is true, and $\Indicator{P}=0$ otherwise.
We write $\delta_s$ for the Dirac measure supported on $s\in S$, so $\delta_s(A)=\Indicator{s\in A}$ for all $A\in\Borel(S)$.
Furthermore, we use comparison functions, in particular, class $\cKi$-functions.
A function $\alpha \colon \Rnn\rightarrow\Rnn$ is in this class if it is continuous, strictly increasing, $\alpha(0)=0$, and $\lim_{r\rightarrow\infty} \alpha(r)=\infty$.

\subsection{Setup}
We consider an \emph{ensemble} of discrete-time time-invariant control systems 
\begin{equation} \label{eq:system}
    x_+ = f(x,u,\theta),
\end{equation}
described by a transition function $f \colon X \times U \times \Theta \rightarrow X$,
where the state and input spaces are metric spaces $(X,d_X)$, $(U,d_U)$, respectively,
and $(\Theta,d_\Theta)$ is a metric space acting as an index set for the ensemble of systems.
Given an initial state $x_0\in X$, a sequence of inputs $\is{u}\in U^N$, $N\in\Np \cup\{\infty\}$, and an ensemble index $\theta\in\Theta$,
a state trajectory $\s_{\is{u},x_0}^\theta(\cdot)$ is defined recursively as usual,
\begin{align*}
    \s_{\is{u},x_0}^\theta(0) & = x_0, \\
     \s_{\is{u},x_0}^\theta(n+1) & = f(\s_{\is{u},x_0}^\theta(n),\is{u}(n),\theta), \: 0 \leq n < N.
\end{align*}
If the initial state is specified by a map $\es_0 \colon \Theta \rightarrow X$ which is clear from context,
we also write $\s_\is{u}^\theta(n) = \s_{\is{u},\es_0(\theta)}^\theta(n)$.
In the following, we consider a stage cost $\ell \colon X \times U \times \Theta \rightarrow \Rnn$, a control horizon $N\in\Np$, and a terminal cost $F_N \colon X\times\Theta\rightarrow\Rnn$.
For an initial state $x_0\in X$, input sequence $\is{u}\in U^N$, and ensemble index $\theta\in\Theta$, we define the corresponding total cost functional by
\begin{equation} \label{eq:dteoc:totalCostFunctional}
    J_N(x_0,u,\theta) = \sum_{n=0}^{N-1} \ell \big( \s_{\is{u},x_0}^\theta(n), \is{u}(n), \theta \big) + F_N\big( \s_{\is{u},x_0}^\theta(N), \theta \big).
\end{equation}
Finally, in the remainder of this work, we work with stage costs of the form
\begin{equation} \label{eq:dteoc:ecc:formOfEll}
    \ell(x,u,\theta)=\ell_u(u) + \ell_0(x,u,\theta),
\end{equation}
with $\ell_u$ and $\ell_0$ nonnegative, which allows us to use fine-grained continuity assumptions on $\ell$.
For brevity, we also define
\begin{equation*}
     J_N^0(x,u,\theta) = \sum_{n=0}^{N-1} \ell_0 \big( \s_{\is{u},x_0}^\theta(n), \is{u}(n), \theta \big) + F_N\big( \s_{\is{u},x_0}^\theta(N), \theta \big).
\end{equation*}
We consider optimal control problems over a whole ensemble of systems, which in turn means aggregating the individual total cost functionals \eqref{eq:dteoc:totalCostFunctional}, indexed by $\theta\in\Theta$.
In this work, we take the average cost with respect to a probability measure on $\Theta$.
From now on, we endow $X$, $U$, and $\Theta$ with their corresponding Borel $\sigma$-algebras.
Before formulating the actual control problem, we have to make sure that the cost functional is well-defined.
\begin{lemma} \label{lem:dteoc:ecc:functionalWelldefined}
If for all $u\in U$, $f(\cdot,u,\cdot)$ is $\Borel(X)\times\Borel(\Theta)$-$\Borel(X)$ measurable
and $\ell_0(\cdot,u,\cdot)$ is $\Borel(X)\times\Borel(\Theta)$-$\Borel(\Rnn)$ measurable,
and $F_N$ is $\Borel(X)\times\Borel(\Theta)$-$\Borel(\Rnn)$ measurable,
and $\es_0 \colon \Theta \rightarrow X$ is $\Borel(\Theta)$-$\Borel(X)$ measurable,
then for all $\is{u}\in U^N$ the integral
\begin{equation}
    \int_{\Theta} J_N(\es_0(\theta),\is{u},\theta) \diff \mu(\theta)
\end{equation}
is well-defined and takes values in $\Rnn\cup\{\infty\}$.
\end{lemma}
\begin{proof}
Let $\is{u}\in U^N$ be arbitrary.
Since compositions of measurable maps are measurable, $\theta \mapsto \ell(\s_\is{u}^\theta(n),\is{u}(n),\theta)$ is $\Borel(\Theta)$-$\Borel(\Rnn)$ measurable for $n=0,\ldots,N-1$, and similarly $\theta \mapsto F_N(\s_\is{u}^\theta(N),\theta)$ is $\Borel(\Theta)$-$\Borel(\Rnn)$ measurable.
The result now follows from the fact that $\ell$ and $F_N$ are nonnegative. \qed
\end{proof}
Assuming for the moment that $f$, $\ell$, $F_N$, and $\es_0  \colon \Theta \rightarrow X$ are measurable,
we can define for $\mu\in\Pbs(\Theta)$ the averaged total cost functional $\bar J_N(\es_0,\cdot,\mu) \colon U^N \rightarrow \Rnn \cup \{\infty\}$ by
\begin{equation} \label{eq:dteoc:ecc:averageCostFunctional}
    \bar J_N(\es_0, \is{u}, \mu) = \int_{\Theta} J_N(\es_0(\theta),u,\theta) \diff \mu(\theta).
\end{equation}
In the following, we are interested in the optimal control problem
\begin{equation} \label{eq:dteoc:ecc:OCP}
    \min_{\is{u}\in U^N} \bar J_N(\es_0, \is{u}, \mu),
\end{equation}
and we define the corresponding value function
\begin{equation}
    \bar V_N(\es_0,\mu) = \inf_{\is{u} \in U^N}\bar J_N(\es_0, \is{u}, \mu).
\end{equation}

\subsection{Continuity properties}
We start with a continuity assumption on the transition function.
\begin{assumption} \label{assump:dteoc:continuityOfTransitionFunction}
There exist $\alpha_x,\alpha_u\in\cKi$ such that for all $\theta\in\Theta$ we have
\begin{equation*}
    d_X \big( f(x,u,\theta), f(x',u',\theta) \big) \leq \alpha_x\big( d_X(x,x') \big) + \alpha_u \big( d_U(u,u') \big)
\end{equation*}
for all $x,x'\in X$ and $u,u' \in U$. 
\end{assumption}
The preceding assumption leads to continuity of the states with respect to input sequences, as made precise in the next result.
\begin{lemma} \label{lem:dteoc:continuityOfStates}
Under Assumption \ref{assump:dteoc:continuityOfTransitionFunction}, for all $M\in\Np$ there exist $\alpha^{[0]},\ldots,\alpha^{[M-1]}\in\cKi$ such that
\begin{equation*}
    d_X\big( \s_{\is{u},x_0}^\theta(M),\s_{\is{u}',x_0}^\theta(M) \big) \leq \sum_{m=0}^{M-1} \alpha^{[m]}\big(d_U(\is{u}(m),\is{u}'(m))\big)
\end{equation*}
for all $x_0\in X$, $\theta \in \Theta$, and $\is{u},\is{u}'\in U^M$.
\end{lemma}
\begin{proof} 
We use induction on $M$. 
For $M=1$, we have
\begin{align*}
     & d_X \big(\s_{\is{u},x_0}^\theta(1), \s_{\is{u}',x_0}^\theta(1) \big) \\
        & \hspace{0.5cm} = d_X\big( f(x_0,\is{u}(0),\theta), f(x_0,\is{u}'(0),\theta) \big) \\
    & \hspace{0.5cm}\leq \alpha_x(0) + \alpha_u\big( d_U(\is{u}(0),\is{u}'(0)) \big) \\
    & \hspace{0.5cm}= \alpha_u \big( d_U(\is{u}(0),\is{u}'(0)) \big),
\end{align*}
so the claim follows with $\alpha^{[0]}=\alpha_u$.

For $M>1$, we have 
\begin{align*}
     & d_X \big(\s_{\is{u},x_0}^\theta(M-1),\s_{\is{u}',x_0}^\theta(M-1) \big) \\
    & \hspace{0.5cm} \leq \sum_{m=0}^{M-2} \tilde\alpha^{[m]} \big(d_U(\is{u}(m),\is{u}'(m)) \big)
\end{align*}
for some $\tilde\alpha^{[0]},\ldots,\tilde\alpha^{[M-2]}\in\cKi$, so
\begin{align*}
     & d_X \big( s_{\is{u},x_0}^\theta(M), s_{\is{u}',x_0}^\theta(M) \big) \\
     & \hspace{0.5cm} =  d_X \big(\s_{\is{u},x_0}^\theta((M-1)+1), \s_{\is{u}',x_0}^\theta((M-1)+1) \big) \\
        & \hspace{0.5cm}= d_X \big(f(\s_{\is{u},x_0}^\theta(M-1),\is{u}(M-1),\theta), \\
            & \hspace{1cm} f(\s_{\is{u}',x_0}^\theta(M-1),\is{u}'(M-1),\theta) \big) \\
        & \hspace{0.5cm}\leq \alpha_x\big(d_X(\s_{\is{u},x_0}^\theta(M-1), \s_{\is{u}',x_0}^\theta(M-1))\big) \\
            & \hspace{1cm} + \alpha_u \big( \is{u}(M-1),\is{u}'(M-1) \big) \\
        & \hspace{0.5cm}\leq \alpha_x\left(\sum_{m=0}^{M-2} \tilde\alpha^{[m]}(d_U(\is{u}(m),\is{u}'(m)) \right) \\
            & \hspace{1cm} + \alpha_u \big( \is{u}(M-1),\is{u}'(M-1) \big),
\end{align*}
where we used Assumption \ref{assump:dteoc:continuityOfTransitionFunction} in the first inequality.
Recalling that for $\alpha\in\cKi$ a weak triangle inequality holds, i.e., $\alpha(s+t)\leq\alpha(2s) + \alpha(2t)$ for all $s,t\in\Rnn$, cf. \cite[Equation~(8)]{kellett2014compendium}, then
\begin{align*}
    & \alpha_x\left(\sum_{m=0}^{M-2} \tilde\alpha^{[m]}(d_U(\is{u}(m),\is{u}'(m)) \right) \\
    & \hspace{0.5cm} \leq \sum_{m=0}^{M-2} \alpha_x\left(2^{n+1}\tilde\alpha^{[m]}(d_U(\is{u}(m),\is{u}'(m)))\right).
\end{align*}
Moreover, since $s \mapsto \alpha(c s) \in \cKi$ for all $c\in\Rp$ and $\alpha\in\cKi$, and since $\alpha_1 \circ \alpha_2 \in \cKi$ for all $\alpha_1,\alpha_2\in\cKi$, cf. again \cite{kellett2014compendium}, 
the result follows with $\alpha^{[m]}=\alpha_x\left(2^{m+1}\tilde\alpha^{[m]}(\cdot)\right)$ for $m=0,\ldots,M-2$,
and $\alpha^{[M-1]}=\alpha_u$.
\qed
\end{proof}
We continue with a regularity assumption on the stage and terminal cost.
Recall that we consider stage costs of the form \eqref{eq:dteoc:ecc:formOfEll}.
\begin{assumption} \label{assump:dteoc:continuityOfLossFunction}
\begin{enumerate}
    \item There exist $\gamma_x,\gamma_u\in\cKi$ such that
        \begin{equation}
        \begin{split}
            |\ell_0(x,u,\theta)-  & \ell_0(x',u',\theta)| \leq \\
            &\gamma_x(d_X(x,x')) + \gamma_u(d_U(u,u'))
        \end{split}
        \end{equation}
        for all $x,x' \in X, u,u' \in U$, and $\theta \in \Theta$.
    \item $\ell_u$ is lower semicontinuous.
    \item There exists $\gamma_N\in\cKi$ such that
        \begin{equation}
            |F_N(x,\theta)-F_N(x',\theta)| \leq \gamma_N(d_X(x,x'))
        \end{equation}
        for all $x,x'\in X$ and $\theta\in\Theta$.
\end{enumerate}
\end{assumption}
We record a simple consequence of Assumptions \ref{assump:dteoc:continuityOfTransitionFunction} and \ref{assump:dteoc:continuityOfLossFunction}.
\begin{lemma} \label{lem:dteoc:ecc:continuityCosts}
Under Assumptions \ref{assump:dteoc:continuityOfTransitionFunction} and \ref{assump:dteoc:continuityOfLossFunction},
for all $M\in\Np$ there exist $\tilde\alpha_0,\ldots,\tilde\alpha_{M-1}\in\cKi$ and $\tilde\gamma_0,\ldots,\tilde\gamma_M\in\cKi$ such that {\small
\begin{align} \label{eq:ecc:continuityStageCosts}
    |\ell_0(M-1) - \ell_0'(M-1)| \leq \sum_{m=0}^{M-1} \tilde\alpha_m\big(d_U(\is{u}(m),\is{u}'(m))\big)
\end{align} }
where $\ell_0(M-1) = \ell_0(\s_{\is{u},x_0}^\theta(M-1),\is{u}(M-1),\theta)$ and $\ell_0'(M-1)=\ell_0(\s_{\is{u}',x_0}^\theta(M-1),\is{u}'(M-1),\theta)$,
and
\begin{align} \label{eq:ecc:continuityFinalCost}
    & |F_N(\s_{\is{u},x_0}^\theta(M),\theta) - F_N(\s_{\is{u}',x_0}^\theta(M),\theta)| \nonumber \\
    & \hspace{0.5cm} \leq \sum_{m=0}^{M} \tilde\gamma_m\big(d_U(\is{u}(m),\is{u}'(m))\big)
\end{align}
for all $x_0\in X$, $\theta \in \Theta$, and $\is{u},\is{u}'\in U^M$.
\end{lemma}
\begin{proof}
We use a case distinction on $M$. For $M=1$ we have
\begin{align*}
 |\ell_0(0) - \ell_0'(0)| & = |\ell_0(x_0,\is{u}(0),\theta)-\ell_0(x_0,\is{u}'(0),\theta)| \\
 & \leq \gamma_u(d_U(\is{u}(0), \is{u}'(0))),
\end{align*}
showing the claim with $\tilde\alpha_0=\gamma_u$.
For $M>1$, we have from Lemma \ref{lem:dteoc:continuityOfStates} that 
\begin{align*}
    & |\ell_0(M-1)-\ell_0(M-1)| \\
    & \hspace{0.5cm} \leq \gamma_x(d_X(\s_{\is{u},x_0}^\theta(M-1), \s_{\is{u}',x_0}^\theta(M-1))) \\
        & \hspace{1cm} + \gamma_u(d_U(\is{u}(M-1),\is{u}'(M-1))) \\
    & \hspace{0.5cm} \leq \gamma_x\left(\sum_{m=0}^{M-1} \alpha^{[m]}(d_U(\is{u}(m),\is{u}'(m))) \right) \\
        & \hspace{1cm} + \gamma_u(d_U(\is{u}(M-1),\is{u}'(M-1))),
\end{align*}
where we first used Assumption \ref{assump:dteoc:continuityOfLossFunction} and then Lemma \ref{lem:dteoc:continuityOfStates}.
As in the proof of Lemma \ref{lem:dteoc:continuityOfStates}, we can now repeatedly use the weak triangle inequality for $\cKi$-functions,
which leads to \eqref{eq:ecc:continuityStageCosts} with $\tilde\alpha_m=\gamma_x(2^{m+1}\alpha^{[m]}(\cdot))$ for $m=0,\ldots,M-2$ and $\tilde\alpha_{M-1}=\gamma_u$.
Finally, an analogous argument leads to \eqref{eq:ecc:continuityFinalCost}.
\qed
\end{proof}
\subsection{Background from calculus of variations}
We first recall the general notion of $\Gamma$-convergence. For further details, we recommend the monograph \cite{dalmaso1993introduction}.
\begin{definition} \label{def:Gamma-conv}
    Let $(Z,d_Z)$ be a metric space, and let us consider a sequence of functionals %
    $\mathcal{F}_k\colon Z \to \R \cup \{ +\infty\}$, $k\in\Np$.
    The sequence $(\mathcal{F}_k)_{k\geq 1}$ is said to \emph{$\Gamma$-converge to a functional $\mathcal{F} \colon Z \to \R \cup \{ +\infty\}$ with respect to the $d_Z$-topology} if the following two conditions are satisfied:
    \begin{enumerate}
        \item \underline{$\liminf$-condition} For every $z \in Z$ and for every sequence $(z_k)_{k\geq 1}$ such that $d_Z(z_k,z)\to 0$, we have that
        \begin{equation*}
            \mathcal{F}(z) \leq \liminf_{k} \mathcal{F}_k(z_k);
        \end{equation*}
        \item \underline{$\limsup$-condition} For every $z \in Z$ there exists a sequence $(\bar z_k)_{k\geq 1}$ (also called \emph{recovery sequence}) such that $d_Z(\bar z_k,z)\to 0$ and 
        \begin{equation*}
            \mathcal{F}(z) \geq \limsup_{k} \mathcal{F}_k(z_k).
        \end{equation*}
    \end{enumerate}
\end{definition}
We will also use the direct method of the calculus of variations, cf. \cite[Chapter~1]{dalmaso1993introduction} for more background.
This involves in particular the notion of lower semicontinuity.
\begin{definition}
Let $(Z,d_Z)$ be a metric space and consider a function $g \colon Z \rightarrow \R \cup \{\infty\}$.
We say that $g$ is lower semicontinuous if for all sequences $(z_k)_k\subseteq Z$ such that $\lim_k z_k = z \in Z$, it holds that
\begin{equation}
    g(z) \leq \liminf_k g(z_k).
\end{equation}
\end{definition}
Note that continuity implies lower semicontinuous.
Furthermore, we also need the notion of coercivity.
\begin{definition}
Let $(Z,d_Z)$ be a metric space and consider a function $g: Z \rightarrow \R \cup \{\infty\}$.
We say that $g$ is coercive if for all $t\in\R$, the set $\{z \in Z \mid g(z) < t\}$ is precompact.
\end{definition}
It is clear from this definition that sums of coercive functions are coercive,
and that if a function is lower bounded by a coercive function, it is also coercive.
The next example provides more concrete instances of this definition.
\begin{example} \label{ex:dteco:ecc:coercivity}
If $(Z,d_Z)$ is a locally-compact metric space and $z_0 \in Z$, then $z \mapsto C \big(d_Z(z,z_0)\big)^q$ is coercive for all $C,q\in\Rp$.
In particular, since $(\R^d, \|\cdot\|_p)$ is locally-compact for all $p\in\Rp$, where $\|z\|_p=\sqrt[p]{|v_1|^p + \ldots +|v_d|^p}$, 
the map $\R^d \ni z \mapsto C \|z\|_p^q$ is coercive for all $C,q,p\in\Rp$.
\end{example}

\section{Existence of minimisers} \label{sec:ecc:existenceMinimisers}
Our next goal is to show that the optimal control problem \eqref{eq:dteoc:ecc:OCP} admits minimisers.
We first need a mild measurability assumption.
\begin{assumption} \label{assump:dteoc:ecc:measurability}
For all $u \in U$, $\ell_0(\cdot,u,\cdot)$ is $\Borel(X)\times\Borel(\Theta)$-$\Borel(\Rnn)$ measurable and $f(\cdot,u,\cdot)$ is $\Borel(X)\times\Borel(\Theta)$-$\Borel(X)$ measurable,
and $F_N$ is $\Borel(X)\times\Borel(\Theta)$-$\Borel(\Rnn)$ measurable.
\end{assumption}
\begin{remark}
Since we work with Borel $\sigma$-algebras, Assumption \ref{assump:dteoc:ecc:measurability} is fulfilled if $f,\ell_0$, and $F_N$ are continuous.
\end{remark}
The first ingredient for the direct method of the calculus of variations is lower semicontinuity, which we establish in the next result.
\begin{lemma} \label{lem:dteoc:lscOfCostFunctional}
Under Assumptions \ref{assump:dteoc:continuityOfTransitionFunction}, \ref{assump:dteoc:continuityOfLossFunction}, and \ref{assump:dteoc:ecc:measurability},
for all $\mu \in \Pbs(\Theta)$ and measurable $\es_0 \colon \Theta\rightarrow X$, 
$U^N \ni \is{u} \mapsto \bar J_N(\es_0, \is{u} ,\mu)$ is well-defined and lower semicontinuous w.r.t.~the product topology on $U^N$.
\end{lemma}
\begin{proof}
Assumptions  \ref{assump:dteoc:continuityOfTransitionFunction}, \ref{assump:dteoc:continuityOfLossFunction}, and \ref{assump:dteoc:ecc:measurability} together ensure that the measurability requirements of Lemma \ref{lem:dteoc:ecc:functionalWelldefined} are fulfilled, and hence the average cost functional is well-defined.
Let us now fix $\is{u}^\ast \in U^N$ and $(\is{u}^{(k)})_{k\in\Np} \subseteq U^N$ such that $\is{u}^{(k)}\rightarrow \is{u}^\ast$ in the product topology on $U^N$.
Observe that for arbitrary $\is{u}\in U^N$, we have
\begin{align*}
    \bar J_N(\es_0,\is{u},\mu) & = \int_{\Theta} J_N(\es_0(\theta),\is{u},\theta) \diff \mu(\theta) \\
    & = \int_\Theta F_N\big( \s_{\is{u}}^\theta(N), \theta \big) \diff \mu(\theta) 
        + \sum_{n=0}^{N-1} \ell_u(\is{u}(n)) \\
        & \hspace{1cm} + \sum_{n=0}^{N-1} \int_\Theta \ell_0 \big( \s_{\is{u}}^\theta(n), \is{u}(n), \theta \big) \diff \mu(\theta),
\end{align*}
where we used the definitions of the cost functionals in the first two equalities,
the form \eqref{eq:dteoc:ecc:formOfEll} of $\ell$, and the fact that $\mu$ is a probability measure.
Since $\ell_u$ is lower semicontinuous by Assumption~\ref{assump:dteoc:continuityOfLossFunction}, we have 
\begin{align*}
    \sum_{n=0}^{N-1} \ell_u(\is{u}^\ast(n)) & \leq \sum_{n=0}^{N-1} \liminf_{k} \ell_u(\is{u}^{(k)}(n)) \\
        & \leq \liminf_{k} \sum_{n=0}^{N-1} \ell_u(\is{u}^{(k)}(n)).
\end{align*}
Next, for brevity define $\ell_0^\ast(n) =  \ell_0 \big( \s_{\is{u}^\ast}^\theta(n), \is{u}^\ast(n), \theta \big)$ and $\ell_0^{(k)}(n) =  \ell_0 \big( \s_{\is{u}^{(k)}}^\theta(n), \is{u}^{(k)}(n), \theta \big) $.
From Lemma \ref{lem:dteoc:ecc:continuityCosts}, we get that for all $n=0,\ldots,N-1$, there exist $\gamma_0,\ldots,\gamma_{n}\in\cKi$ such that
\begin{equation*}
    \big| \ell_0^\ast(n) - \ell_0^{(k)}(n) \big| \leq \sum_{m=0}^{n} \gamma_m \big(d_U(\is{u}^\ast(m),\is{u}^{(k)}(m)) \big)
\end{equation*}
which implies that for all $n=1,\ldots,N-1$,
\begin{align*}
    & \lim_{k} \sup_{\theta\in\Theta} 
        \Big| \ell_0^{(k)}(n)  - \ell_0^\ast(n) \Big|
     \\
    & \hspace{0.5cm} \leq \lim_{k} \sup_{\theta\in\Theta} 
     \sum_{m=0}^{n-1} \gamma_m \big(d_U(\is{u}^{(k)}(m),\is{u}^\ast(m)) \big) \\
    & \hspace{0.5cm} = \lim_{k} \sum_{m=0}^{n-1} \gamma_m \big(d_U(\is{u}^{(k)}(m),\is{u}^\ast(m)) \big) \\
    & \hspace{0.5cm} = \sum_{m=0}^{n-1} \gamma_m \big( \lim_{k}  d_U (\is{u}^{(k)}(m),\is{u}^\ast(m)) \big) \\
    & \hspace{0.5cm} = 0.
\end{align*}
This shows that for all $n=0,\ldots,N-1$, we have $\ell_0 \big( \s_{\is{u}^{(k)}}^\theta(n), \is{u}^{(k)}(n), \theta \big)  \rightarrow  \ell_0 \big( \s_{\is{u}^\ast}^\theta(n), \is{u}^\ast(n), \theta \big)$ uniformly in $\theta$, and, in particular, for $\mu$-a.e.~$\theta \in \Theta$.
Hence, by Fatou's Lemma, we conclude that
\begin{align*}
    & \int_\Theta \ell \big( \ell_0 \big( \s_{\is{u}^\ast}^\theta(n), \is{u}^\ast(n), \theta \big) \diff \mu(\theta) \\
    & \hspace{0.5cm} \leq \liminf_{k} \int_\Theta \ell_0 \big( \s_{\is{u}^{(k)}}^\theta(n), \is{u}^{(k)}(n), \theta \big) \diff \mu(\theta).
\end{align*}
Using a similar reasoning, we also find that
\begin{align*}
      & \int_\Theta F_N \big( \s_{\is{u}^\ast}^\theta(N), \theta \big) \diff \mu(\theta) \\
      & \hspace{0.5cm} \leq \liminf_{k}  \int_\Theta F_N \big( \s_{\is{u}^{(k)}}^\theta(N), \theta \big) \diff \mu(\theta).
\end{align*}
Altogether, we find that
\begin{align*}
    \bar J_N(\es_0, \is{u}^\ast, \mu) \leq \liminf_{k} \bar J_N(\es_0,\is{u}^{(k)},\mu),
\end{align*}
establishing lower semicontinuity of $\is{u} \mapsto \bar J_N(\es_0,\is{u},\mu)$. \qed
\end{proof}
The second ingredient for the direct method of calculus of variations is coercivity.
The following assumption is sufficient for this.
\begin{assumption} \label{assump:dteoc:coercivityOfStageCost}
There exists a function $r\colon U \rightarrow \Rnn$ that is coercive, and $\ell(x,u,\theta)\geq r(u)$ holds for all $x\in X$, $u \in U$, $\theta \in \Theta$.
\end{assumption}
Since this assumption is somewhat abstract, we provide some concrete instances in the following example.
\begin{example}
    Since according to \eqref{eq:dteoc:ecc:formOfEll} we have $\ell(x,u,\theta)= \ell_0(x,u,\theta) + \ell_u(u)$, where $\ell_0, \ell_u$ take non-negative values,
    Assumption~\ref{assump:dteoc:coercivityOfStageCost} holds whenever $\ell_u$ is coercive.
    Recalling Example \ref{ex:dteco:ecc:coercivity}, if $(U,d_U)$ is a locally-compact metric space, this is the case if we choose $v_0 \in U$ and we set $\ell_u(v) = \lambda \big(d_U(v,v_0)\big)^q$, with $\lambda,q>0$ positive parameters. 
    This is essentially the setting considered in \cite{scagliotti2023optimal}.
    For example, if $(U,d_U)=(\R^d, \|\cdot \|_p)$, then by defining $\ell_u(v) :=  \lambda \|v\|_p^q$ with $\lambda, q>0$ and choosing $\ell_0$ complying with Assumption~\ref{assump:dteoc:continuityOfLossFunction}, it turns out that the coercivity condition of Assumption~\ref{assump:dteoc:coercivityOfStageCost} is satisfied.
\end{example}
We can now establish coercivity of the average cost functional.
\begin{lemma} \label{lem:dteoc:coercivityOfCostFunctional}
Under Assumptions \ref{assump:dteoc:continuityOfTransitionFunction}, \ref{assump:dteoc:continuityOfLossFunction}, \ref{assump:dteoc:ecc:measurability}, and \ref{assump:dteoc:coercivityOfStageCost}, for all $\mu\in\Pbs(\Theta)$ and measurable $\es_0 \colon \Theta\rightarrow X$,
the map $U^N \ni \is{u} \mapsto \bar J_N(\es_0,\is{u},\mu)$ is well-defined and coercive w.r.t. the product topology on $U^N$.
\end{lemma}
\begin{proof}
As in Lemma \ref{lem:dteoc:coercivityOfCostFunctional}, Assumptions \ref{assump:dteoc:continuityOfTransitionFunction}, \ref{assump:dteoc:continuityOfLossFunction}, and  \ref{assump:dteoc:ecc:measurability} ensure that the map is well-defined.
Since $r$ from Assumption \ref{assump:dteoc:coercivityOfStageCost} is coercive, also the function $U^N \ni \is{u} \mapsto \sum_{n=0}^{N-1} r(\is{u}(n))$ is coercive w.r.t. the product topology on $U^N$.
Furthermore, we have
\begin{align*}
    \bar J_N(\es_0, \is{u}, \mu) & = \int_\Theta J_N(\es_0(\theta),\is{u},\theta) \diff \mu(\theta) \\
    & \geq \int_\Theta \sum_{n=0}^{N-1} \ell \big( \s_{\is{u}}^\theta(n), \is{u}(n), \theta \big) \diff \mu(\theta) \\
    & \geq \int_\Theta \sum_{n=0}^{N-1} r(\is{u}(n)) \diff \mu(\theta) \\
    & = \sum_{n=0}^{N-1} r(\is{u}(n)),
\end{align*}
where we used the definition of the cost functionals, the fact that $F_N$ is nonnegative, Assumption \ref{assump:dteoc:coercivityOfStageCost}, and that $\mu$ is a probability measure.
Altogether, this shows that $U^N \ni \is{u} \mapsto \bar J_N(\es_0,\is{u},\mu)$ is well-defined and coercive. \qed
\end{proof}
We are finally ready to show that the optimal control problem \eqref{eq:dteoc:ecc:OCP} admits minimisers. 
\begin{theorem} \label{thm:dteoc:existenceOfMinimisers}
Under Assumptions \ref{assump:dteoc:continuityOfTransitionFunction}, \ref{assump:dteoc:continuityOfLossFunction}, \ref{assump:dteoc:ecc:measurability}, and \ref{assump:dteoc:coercivityOfStageCost}, for all $\mu \in \Pbs(\Theta)$ and measurable $\es_0\colon \Theta\rightarrow X$,
the map $U^N \ni \is{u} \mapsto \bar J_N(\es_0,\is{u},\mu)$ is well-defined and there exists $\is{u}^\ast \in U^N$ with
\begin{equation}
    \bar J_N(\es_0, \is{u}^\ast, \mu) = \bar V_N(\es_0,\mu) = \min_{\is{u} \in U^N} \bar J_N(\es_0,\is{u},\mu).
\end{equation}
\end{theorem}
\begin{proof}
Combining Lemma \ref{lem:dteoc:lscOfCostFunctional} and \ref{lem:dteoc:coercivityOfCostFunctional} shows that $\bar J_N(\es_0, \cdot, \mu)$ is well-defined, lower semicontinuous and coercive, so the claim follows from the direct method of the calculus of variations, cf. \cite[Theorem~1.15]{dalmaso1993introduction}. \qed
\end{proof}

\section{Approximation via $\Gamma$-convergence} \label{sec:ecc:gammaConv}
For a generic $\mu\in\Pbs(\Theta)$, solving \eqref{eq:dteoc:ecc:OCP} is in general intractable due to the integral in the definition of \eqref{eq:dteoc:ecc:averageCostFunctional}.
However, we can simplify the problem by approximating $\mu$ by a more convenient probability measure.
For example, if we use an empirical measure $\hat\mu=\frac1M \sum_{i=1}^M \delta_{\theta_i}$,
problem \eqref{eq:dteoc:ecc:OCP} becomes
\begin{equation} \label{eq:dteoc:ecc:approxOCP}
    \min_{\is{u}\in U^N} \frac1M \sum_{i=1}^M J_N(\es_0(\theta_i),\is{u},\theta_i),
\end{equation}
which is closer to standard problems considered in numerical optimal control.
This leads to the question how good the solution of \eqref{eq:dteoc:ecc:approxOCP} is for the original problem \eqref{eq:dteoc:ecc:OCP}.
As in \cite{scagliotti2023optimal}, we tackle this question using $\Gamma$-convergence.
\begin{remark}
We briefly discuss the existence and construction of empirical probability measures approximating a given Borel probability measure.
If $\Theta$ is separable, which is in particular the case if $\Theta$ is compact, then Varadarajan's Theorem, cf. \cite[Theorem~11.4.1]{dudley2002real}, guarantees that for all $\mu\in\Pbs(\Theta)$ there exists $(\mu_k)_k\subseteq\Pbs(\Theta)$, each $\mu_k$ having only finitely many atoms, with $\mu_k \weak \mu$.
More precisely, if $(X_k)_k$ is a sequence of $\Theta$-valued random variables that are independent and identically distributed according to $\mu$, then $\mu$-a.s. the probability measures $\mu_k = \frac1k \sum_{m=1}^k \delta_{X_k}$ fulfill $\mu_k \weak \mu$.
Note that this allows a data-driven approach, where $\mu$ is unknown, but independent and identically distributed samples thereof are available.
\end{remark}
Consider $\mu \in \Pbs(\Theta)$ and a sequence $(\mu_k)_k\subseteq\Pbs(\Theta)$ with $\mu_k\weak\mu$, as well as a measurable $\es_0: \Theta\rightarrow X$.
Define $\mathcal{F}=\bar J_N(\es_0,\cdot,\mu)$ and $\mathcal{F}_k=\bar J_N(\es_0,\cdot,\mu_k)$, and we interpret $\mathcal{F}_k$ as a sequence of approximations of $\mathcal{F}$.
In this section, we will establish the $\Gamma$-convergence of $\mathcal{F}_k$ to $\mathcal{F}$, and deduce convergence of the minimisers.
\subsection{Assumptions}
Compared to merely establishing existence of minimisers, we need slightly stronger continuity assumptions.
\begin{assumption} \label{assump:dteoc:additionalContinuity}
The maps $\es_0$, $f$, $\ell_0$, and $F_N$ are continuous.
\end{assumption}
\begin{remark} 
Assumption \ref{assump:dteoc:additionalContinuity} implies Assumption \ref{assump:dteoc:ecc:measurability}.
\end{remark}
Before moving on, we provide some intuition for continuity conditions contained in Assumption~\ref{assump:dteoc:additionalContinuity}. 
Let us consider $\theta_\infty \in \Theta$ and a sequence $(\theta_k)_k\subseteq \Theta$ such that $d_\Theta(\theta_k, \theta_\infty )\to 0$. 
If we set $\mu_k = \delta_{\theta_k}$ for every $k\in \Np$ and $\mu = \delta_{\theta_\infty}$, then from a direct computation it follows that $\mu_k \weak \mu $.
In such a scenario, it turns out that $\bar J_N(\bar x_0, \cdot , \mu) = J_N(\bar x_0(\theta_\infty), \cdot , \theta_\infty)$, as well as $\bar J_N(\bar x_0, \cdot , \mu_k) = J_N(\bar x_0(\theta_k), \cdot , \theta_k)$ for every $k \in \Np$. 
Hence, in view of establishing a convergence result of $u^*_k \in \arg \min_{\is{u}\in U^N} J_N(\bar x_0(\theta_k), \is{u}, \theta_k)$ towards $\arg \min_{\is{u}\in U^N} J_N(\bar x_0(\theta_\infty), \is{u}, \theta_\infty)$, the joint continuity in $\theta$ and $\is{u}$ of $(\theta, \is{u})\mapsto J_N(\bar x_0(\theta), \is{u}, \theta)$ plays a pivotal role.
In this regard, Assumption~\ref{assump:dteoc:additionalContinuity} provides sufficient conditions.

Next, we need the following technical assumption.
\begin{assumption} \label{assump:dteoc:convergenceOfIntegralTerms}
Consider the $\Gamma$-convergence setup from above under Assumptions \ref{assump:dteoc:continuityOfTransitionFunction}, \ref{assump:dteoc:continuityOfLossFunction}, and \ref{assump:dteoc:additionalContinuity}.
$\es_0$, $f$, $\ell_0$, and $F_N$, as well as $(\mu_k)_k$ and $\mu$, are such that for all $\is{u}\in U^N$ we have for $n=0,\ldots,N-1$ that
\begin{align*}
    & \lim_k \int_\Theta \ell_0(\s_{\is{u}}^\theta(n),\is{u}(n),\theta) \diff \mu_k(\theta) \nonumber \\
    & \hspace{0.5cm} = \int_\Theta \ell_0(\s_{\is{u}}^\theta(n),\is{u}(n),\theta) \diff \mu(\theta)
\end{align*}
and
\begin{align*}
    \lim_k\int_\Theta F_N(\s_{\is{u}}^\theta(N),\theta) \diff \mu_k(\theta) 
    = \int_\Theta F_N(\s_{\is{u}}^\theta(N),\theta) \diff \mu(\theta).
\end{align*}
\end{assumption}
We now provide more concrete conditions that ensure that Assumption \ref{assump:dteoc:convergenceOfIntegralTerms} holds.
For convenience, define for $n=0,\ldots,N-1$
\begin{equation}
 \varphi_n \colon \Theta \rightarrow \Rnn, \: \varphi_n(\theta)= \ell_0(\s_{\is{u}}^\theta(n),\is{u}(n),\theta)
\end{equation}
and
\begin{equation}
    \varphi_N \colon \Theta\rightarrow \Rnn, \: \varphi_N(\theta)=F_N(\s_{\is{u}}^\theta(N),\theta).
\end{equation}
The first concrete condition is essentially the one used in \cite{scagliotti2023optimal}, cf. the proof of Lemma 3.1 in this reference.
\begin{lemma}
Consider the $\Gamma$-convergence setup from above under Assumptions \ref{assump:dteoc:continuityOfTransitionFunction}, \ref{assump:dteoc:continuityOfLossFunction}, and \ref{assump:dteoc:additionalContinuity}.
If $\Theta$ is compact, then Assumption \ref{assump:dteoc:convergenceOfIntegralTerms} holds.
\end{lemma}
\begin{proof}
Observe that $\varphi_0,\varphi_1,\ldots,\varphi_N$ are all continuous functions since they are compositions of continuous functions.
Since $\Theta$ is compact, $\varphi_0,\varphi_1,\ldots,\varphi_N$ are all bounded,
and hence $\int_\Theta \varphi_n(\theta)\diff\mu_k(\theta)\rightarrow\int_\Theta \varphi_n(\theta)\diff\mu(\theta)$ follows by the definition of weak convergence of probability measures. \qed
\end{proof}
The next result states that uniform integrability of $\varphi_0,\varphi_1,\ldots,\varphi_N$ is enough to ensure that Assumption \ref{assump:dteoc:convergenceOfIntegralTerms} holds.
This is essentially the most general condition possible, cf. \cite[Lemma~5.11]{kallenberg2021foundations}. 
\begin{lemma}
Consider the $\Gamma$-convergence setup from above under Assumptions \ref{assump:dteoc:continuityOfTransitionFunction}, \ref{assump:dteoc:continuityOfLossFunction}, and \ref{assump:dteoc:additionalContinuity}.
If for $n=0,\ldots,N$ we have that
\begin{equation}
     \lim_{M\rightarrow\infty} \limsup_{k} \int_\Theta \varphi_n(\theta) \Indicator{\varphi_n(\theta)>M} \diff \mu_k(\theta) = 0,
\end{equation}
then Assumption \ref{assump:dteoc:convergenceOfIntegralTerms} holds.
\end{lemma}
\begin{proof}
Consider $\Theta$-valued random variables $\zeta_k$, $\zeta$ with law $\mu_k$ and $\mu$, respectively, so that by construction, $\zeta_k \rightarrow \zeta$ in distribution.
For $n=0,\ldots,N$, define $\xi_k=\varphi_n(\zeta)$ and $\xi=\varphi_n(\zeta)$. 
Since $\varphi_n$ is continuous, by the continuous mapping theorem we get that also $\xi_k \rightarrow \xi$ in distribution.
The condition of the result now becomes
\begin{equation*}
    \lim_{M\rightarrow\infty} \limsup_{k} \E[\xi_k \Indicator{\xi_k >M}] = 0,
\end{equation*}
which means that $(\xi_k)_k$ is uniformly integrable,
which in turn implies that
\begin{equation*}
    \int_\Theta \varphi_n(\theta) \diff \mu_k(\theta) = \E[\xi_k]
    \rightarrow
    \E[\xi] = \int_\Theta \varphi_n(\theta) \diff \mu(\theta),
\end{equation*}
which shows that Assumption \ref{assump:dteoc:convergenceOfIntegralTerms} is fulfilled. \qed
\end{proof}
\subsection{Convergence results}
We are now ready to state the $\Gamma$-convergence result.
\begin{theorem} \label{thm:dteoc:GammaConvergence}
Under Assumptions \ref{assump:dteoc:continuityOfTransitionFunction}, \ref{assump:dteoc:continuityOfLossFunction}, \ref{assump:dteoc:additionalContinuity}, and \ref{assump:dteoc:convergenceOfIntegralTerms}, 
the sequence of functionals $\left( \bar J_N(\bar x_0, \cdot, \mu_k) \right)_{k\geq 1}$ is $\Gamma$-convergent to the functional $U^N\ni u \mapsto \bar J_N(\bar x_0, u, \mu)$ with respect to the product topology of $U^N$.
\end{theorem}
\begin{proof}
\underline{$\liminf$-condition} Consider $\is{u}^\ast\in U^N$ and $(\is{u}^{(k)})_k\subseteq U^N$ with $\is{u}^{(k)}\rightarrow u$ in the product topology on $U^N$.
We have
\begin{align*}
    \bar J_N(\es_0,\is{u}^\ast,\mu) & = \int_{\Theta} J_N(\es_0(\theta),\is{u}^\ast,\theta) \diff \mu(\theta) \\
    & = \int_\Theta F_N\big( \s_{\is{u}^\ast}^\theta(N), \theta \big) \diff \mu(\theta) 
        + \sum_{n=0}^{N-1} \ell_u(\is{u}^\ast(n)) \\
        & \hspace{1cm} + \sum_{n=0}^{N-1} \int_\Theta \ell_0 \big( \s_{\is{u}^\ast}^\theta(n), \is{u}^\ast(n), \theta \big) \diff \mu(\theta),
\end{align*}
and we consider each of the three remaining terms in turn.
Since $\ell_u$ is lower semicontinuous, we get
\begin{align*}
    \sum_{n=0}^N \ell_u(\is{u}^\ast(n)) & \leq \sum_{n=0}^N \liminf_k \ell_u(\is{u}^{(k)}(n)) \\
        & \leq \liminf_k  \sum_{n=0}^N \ell_u(\is{u}^{(k)}(n)).
\end{align*}
Next, for brevity define 
$\ell_0^\ast(n;\theta) =  \ell_0 \big( \s_{\is{u}^\ast}^\theta(n), \is{u}^\ast(n), \theta \big)$ 
and 
$\ell_0^{(k)}(n;\theta) =  \ell_0 \big( \s_{\is{u}^{(k)}}^\theta(n), \is{u}^{(k)}(n), \theta \big) $ for $n=0,\ldots,N-1$ and $\theta\in\Theta$.
As in the proof of Lemma \ref{lem:dteoc:lscOfCostFunctional}, there are $\tilde \gamma_0,\ldots,\tilde \gamma_n\in\cKi$ with
$|\ell_0^\ast(n;\theta)-\ell_0^{(k)}(n;\theta)|\leq \sum_{m=0}^n \tilde \gamma_m(d_U(\is{u}^\ast(m), \is{u}^{(k)}(m)) =: R^{(k)}_n$,
and since $\is{u}^{(k)}\rightarrow \is{u}^\ast$ in the product topology, the right-hand side converges to zero for $k\rightarrow\infty$.
For $k\in\Np$ we now have
\begin{align*}
     \int_\Theta \ell_0^\ast(n;\theta) \diff \mu_k(\theta)
     \leq \int_\Theta\ell_0^{(k)}(n;\theta) \diff \mu_k(\theta) + R^{(k)}_n
\end{align*}
and a fortiori
\begin{align*}
    &  \int_\Theta \ell_0^\ast(n;\theta) \diff \mu_k(\theta)
        \leq \liminf_k \int_\Theta \ell_0^\ast(n;\theta) \diff \mu_k(\theta) \\
    & \hspace{0.5cm} \leq \liminf_k \int_\Theta \ell_0^{(k)}(n;\theta) \diff \mu_k(\theta)
        + R^{(k)}_n \\
    & \hspace{0.5cm} =  \liminf_k \int_\Theta \ell_0^{(k)}(n;\theta) \diff \mu_k(\theta),
\end{align*} 
where in the first step we used Fatou's lemma for distributional convergence of nonnegative random variables, cf. \cite[Lemma~5.11]{kallenberg2021foundations},
which is applicable due to the continuity properties from Assumptions \ref{assump:dteoc:continuityOfTransitionFunction}, \ref{assump:dteoc:coercivityOfStageCost}, and \ref{assump:dteoc:additionalContinuity} and the fact that $\mu_k \weak \mu$.

Finally, we can use the same arguments again to get that
$    \int_\Theta F_N\big( \s_{\is{u}^\ast}^\theta(N), \theta \big) \diff \mu(\theta) 
     \leq
    \liminf_k \int_\Theta F_N\big( \s_{\is{u}^{(k)}}^\theta(N), \theta \big) \diff \mu_k(\theta)$.
Combining everything, we get that
\begin{align*}
    & \bar J_N(\es_0, u, \mu) \leq  \liminf_k \int_\Theta F_N\big( \s_{\is{u}^{(k)}}^\theta(N), \theta \big) \diff \mu_k(\theta) \\
        & \hspace{1cm} +  \sum_{n=0}^{N-1} \liminf_k \int_\Theta  \ell_0 \big( \s_{\is{u}^{(k)}}^\theta(n), \is{u}^{(k)}(n), \theta \big) \diff \mu_k(\theta) \\
        & \hspace{1cm} + \liminf_k  \sum_{n=0}^N \ell_u(\is{u}^{(k)}(n)) \\
    & \hspace{0.5cm} \leq \liminf_k \int_\Theta J_N(\es_0(\theta),u_k,\theta) \diff \mu_k(\theta) \\
    & \hspace{0.5cm} = \liminf_k \bar J_N(\es_0, u_k, \mu_k),
\end{align*}
establishing the $\liminf$-inequality.

\underline{$\limsup$-condition} Let $\is{u}^\ast \in U^N$ be arbitrary. We define a sequence $(\is{u}^{(k)})_k\subseteq U^N$ by setting $\is{u}^{(k)}=\is{u}^\ast$, so trivially $\is{u}^{(k)}\rightarrow \is{u}^\ast$ in the product topology on $U^N$.
Observe that
\begin{align*}
    & \bar J_N(\es_0,\is{u}^{(k)},\mu_k) = \int_\Theta F_N(\s_{\is{u}^\ast}^\theta(N),\theta) \diff \mu_k(\theta) \\
        & \hspace{0.5cm}   + \sum_{n=0}^{N-1} \ell_u(\is{u}^\ast(n)) + \sum_{n=0}^{N-1} \int_\Theta \ell_0(\s_{\is{u}^\ast}^\theta(n),\is{u}^\ast(n),\theta) \diff \mu_k(\theta),
\end{align*}
and from Assumption \ref{assump:dteoc:convergenceOfIntegralTerms} we have that, for $k\to \infty$, the integral
$\int_\Theta \ell_0(\s_{\is{u}^\ast}^\theta(n),\is{u}^\ast(n),\theta)  \diff \mu_k(\theta)$ tends to $  \int_\Theta \ell_0(\s_{\is{u}^\ast}^\theta(n),\is{u}^\ast(n),\theta)  \diff \mu(\theta)$ for all $n=0,\ldots,N-1$,
and $\int_\Theta F_N(\s_{\is{u}^\ast}^\theta(N),\theta) \diff \mu_k(\theta)
    \to 
    \int_\Theta F_N(\s_{\is{u}^\ast}^\theta(N),\theta) \diff \mu(\theta)$.
Altogether, we obtain that $\limsup_k \bar J_N(\es_0,\is{u}_k,\mu_k) = \lim_k \bar J_N(\es_0,\is{u}^\ast,\mu_k) = \bar J_N(\es_0,\is{u}^\ast,\mu)$,
showing the $\limsup$-condition. \qed
\end{proof}
\begin{remark}
Inspecting the preceding proof reveals that we needed Assumption \ref{assump:dteoc:convergenceOfIntegralTerms} only for the $\limsup$-condition,
since for the $\liminf$-condition, we could use Fatou's lemma for distributional convergence of nonnegative random variables, which does not require any additional assumptions (like boundedness, or more generally uniform integrability) on the test functions.
\end{remark}
As an application of Theorem \ref{thm:dteoc:GammaConvergence}, we show convergence of minimisation problems and minimisers.
\begin{corollary} \label{cor:ecc:convOfMin}
Consider the situation of Theorem \ref{thm:dteoc:GammaConvergence}, and let in addition Assumption \ref{assump:dteoc:coercivityOfStageCost} hold.
We have
\begin{equation}
    \inf_{\is{u} \in U^N} \bar J_N(\es_0,\is{u},\mu) = \lim_k \inf_{\is{u} \in U^N} \bar J_N(\es_0, \is{u}, \mu_k).
\end{equation}
Furthermore, for $k\in\Np$, let $\is{u}^{(k)} \in U^N$ be any minimiser of $\is{u} \mapsto \bar J_N(\es_0,\is{u},\mu_k)$ (there always exists one due to Theorem \ref{thm:dteoc:existenceOfMinimisers}).
Then, the sequence $(\is{u})_k$ is pre-compact and, if any limiting point $\is{u}^\ast\in U^N$ is a minimiser of $\is{u}\mapsto \bar J_N(\es_0,\is{u},\mu)$.
\end{corollary}
\begin{proof}
Due to the $\Gamma$-convergence from Theorem \ref{thm:dteoc:GammaConvergence}, this result follows immediately from \cite[Corollary~7.20]{dalmaso1993introduction}. \qed
\end{proof}
\section{Conclusion} \label{sec:ecc:conclusion}
In this work, we provide the first steps toward optimal control of ensembles of discrete-time dynamical systems.
We consider very general nonlinear systems, our optimal control setup encompasses a broad class of stage costs and terminal costs, and we focus on average total cost functions based on arbitrary Borel probability measures over the ensembles.
We were able to prove existence of minimisers of the corresponding ensemble optimal control problem, and we established a $\Gamma$-convergence result, both under mild assumptions.
The latter ensures the consistent approximation of the ensemble optimal control problem, for example, by empirical probability measures, which in turn allow the application of numerical algorithms.
Altogether, we build a solid theoretical foundation for optimal control of ensembles of discrete-time dynamical systems.
Ongoing work is concerned with aggregation approaches beyond averages, for example, using min-max formulations, as well as generalisations to the infinite-horizon case.

\section*{Acknowledgements}
The authors acknowledge funding from DFG Project FO 767/10-2 (eBer-24-32734) "Implicit Bias in Adversarial Training".

\bibliographystyle{unsrt} 
\bibliography{refs} 

@article{scagliotti2023optimal,
  title={Optimal control of ensembles of dynamical systems},
  author={Scagliotti, Alessandro},
  journal={ESAIM: Control, Optimisation and Calculus of Variations},
  volume={29},
  pages={22},
  year={2023},
  publisher={EDP Sciences}
}

@article{scagliotti2025minimax,
  title={Minimax problems for ensembles of control-affine systems},
  author={Scagliotti, Alessandro},
  journal={SIAM Journal on Control and Optimization},
  volume={63},
  number={1},
  pages={502--523},
  year={2025},
  publisher={SIAM}
}

@article{kellett2014compendium,
  title={A compendium of comparison function results},
  author={Kellett, Christopher M},
  journal={Mathematics of Control, Signals, and Systems},
  volume={26},
  number={3},
  pages={339--374},
  year={2014},
  publisher={Springer}
}

@book{dalmaso1993introduction,
  title={An introduction to $\Gamma$-convergence},
  author={Dal Maso, Gianni},
  year={1993},
  publisher={Springer Science \& Business Media}
}

@book{kallenberg2021foundations,
  title={Foundations of modern probability},
  author={Kallenberg, Olav},
  year={2021},
  edition={3},
  publisher={Springer}
}

@article{li2006control,
  title={Control of inhomogeneous quantum ensembles},
  author={Li, Jr-Shin and Khaneja, Navin},
  journal={Physical Review A—Atomic, Molecular, and Optical Physics},
  volume={73},
  number={3},
  pages={030302},
  year={2006},
  publisher={APS}
}

@article{scagliotti2025ensemble,
  title={Ensemble optimal control for managing drug resistance in cancer therapies},
  author={Scagliotti, Alessandro and Scagliotti, Federico and Locati, Laura Deborah and Sottotetti, Federico},
  journal={arXiv preprint arXiv:2503.08927},
  year={2025}
}

@article{ruths2012optimal,
  title={Optimal control of inhomogeneous ensembles},
  author={Ruths, Justin and Li, Jr-Shin},
  journal={IEEE Transactions on Automatic Control},
  volume={57},
  number={8},
  pages={2021--2032},
  year={2012},
  publisher={IEEE}
}

@book{dudley2002real,
  title={Real analysis and probability},
  author={Dudley, Richard M},
  year={2002},
  edition={2},
  publisher={Cambridge University Press}
}

@article{Ross2015Riemann,
author = {Ross, Michael I. and Proulx, Ronald J. and Karpenko, Mark and Gong, Qi},
title = {Riemann–Stieltjes Optimal Control Problems for Uncertain Dynamic Systems},
journal = {Journal of Guidance, Control, and Dynamics},
volume = {38},
number = {7},
pages = {1251-1263},
year = {2015},
doi = {10.2514/1.G000505},
}

@article{Caillau2018Solving,
author = {Caillau, Jean-Baptiste and Cerf, M and Sassi, A and Trélat, Emmanuel and Zidani, Hasnaa},
title = {Solving chance constrained optimal control problems in aerospace via kernel density estimation},
journal = {Optimal Control Applications and Methods},
volume = {39},
number = {5},
pages = {1833-1858},
doi = {10.1002/oca.2445},
year = {2018},
}

@article{Li2009Ensemble,
  author={Li, Jr-Shin and Khaneja, Navin},
  journal={IEEE Transactions on Automatic Control}, 
  title={Ensemble Control of Bloch Equations}, 
  year={2009},
  volume={54},
  number={3},
  pages={528-536},
  doi={10.1109/TAC.2009.2012983}
}

@article{Baparou2024Conditions, 
title = {Conditions for uniform ensemble output controllability, and obstruction to uniform ensemble controllability},
journal = {Mathematical Control and Related Fields},
volume = {14},
number = {3},
pages = {1128-1175},
year = {2024},
issn = {2156-8472},
doi = {10.3934/mcrf.2023036},
author = {Baparou Danhane and Jérôme Lohéac and Marc Jungers},
}

@article{Loheac2016From,
     author = {Loh\'eac, J\'er\^ome and  Zuazua, Enrique},
     title = {From averaged to simultaneous controllability},
     journal = {Annales de la Facult\'e des sciences de Toulouse : Math\'ematiques},
     pages = {785--828},
     publisher = {Universit\'e Paul Sabatier, Toulouse},
     volume = {25},
     number = {4},
     year = {2016},
     doi = {10.5802/afst.1511},
}

@article{Augier2018Adiabatic,
author = {Augier, Nicolas and Boscain, Ugo and Sigalotti, Mario},
title = {Adiabatic Ensemble Control of a Continuum of Quantum Systems},
journal = {SIAM Journal on Control and Optimization},
volume = {56},
number = {6},
pages = {4045-4068},
year = {2018},
doi = {10.1137/17M1140327},
}

@inproceedings{tie2016controllability,
  title={On controllability of discrete-time linear ensemble systems with linear parameter variation},
  author={Tie, Lin and Li, Jr-Shin},
  booktitle={2016 American Control Conference (ACC)},
  pages={6357--6362},
  year={2016},
  organization={IEEE}
}

\end{document}